\documentclass{amsart}
\usepackage{amsfonts, amsbsy, amsmath, amssymb}

\makeatletter
\@namedef{subjclassname@2010}{%
  \textup{2020} Mathematics Subject Classification}
\makeatother

\ifx\fiverm\undefined 
  \newfont\fiverm{cmr5} 
\fi

\newtheorem{thm}{Theorem}[section]

\newtheorem{exmp}[thm]{Example}

\newtheorem{rmk}[thm]{Remark}

\newtheorem{thm-con}[thm]{Theorem-Conjecture}
\numberwithin{equation}{section}

\theoremstyle{definition}

\allowdisplaybreaks

\newcommand{\f}{\Bbb F}

\begin{document}

\title[radical extension of the field of rational functions]{On a radical extension of the field of rational functions in several variables}

\author[Xiang-dong Hou]{Xiang-dong Hou}
\address{Department of Mathematics and Statistics,
University of South Florida, Tampa, FL 33620}
\email{xhou@usf.edu}

\author[Christopher Sze]{Christopher Sze}
\address{Department of Mathematics and Statistics,
University of South Florida, Tampa, FL 33620}
\email{csze@mail.usf.edu}



\begin{abstract}
Let $F$ be a field and let $F(X_1,\dots,X_n)$ be the field of rational functions in $n$ variables $X_1,\dots,X_n$ over $F$. Let $T=X_1+\cdots+X_n\in F(X_1,\dots,X_n)$ and let $m$ be a positive integer such that $\text{char}\,F\nmid m$. Is it possible to express each $X_i$ as a rational function in $X_1^m\dots,X_n^m$ and $T$ over $F$? It is not difficult to prove that this can be done but it is another matter to show how this is done. We answer the above question affirmatively with a nonconstructive proof and a constructive proof.
\end{abstract}

\maketitle

\section{Introduction}

While studying permutation rational functions over finite fields, the authors of  \cite{Bartoli-Hou-arXiv2008.03432} found it necessary to express each $X_i$ in the field of rational functions $\f_q(X_1,\dots,X_4)$ as rational function in $X_1^2,\dots, X_4^2$ and $X_1+\cdots+X_4$ over $\f_q$, where $q$ is odd. More generally, consider the following question. Let $F$ be a field and $F(X_1,\dots,X_n)$ be the field of rational functions in $n$ variables $X_1,\dots,X_n$ over $F$. Let $T=X_1+\cdots+X_n\in F(X_1,\dots,X_n)$ and let $m$ be a positive integer such that $\text{char}\,F\nmid m$. Is it possible to express each $X_i$ as a rational function in $X_1^m,\dots,X_n^m, T$ over $F$? The answer is positive; a quick nonconstructive proof is given in Section~2. Now that we know that it can be done, the next question how it is done. This becomes more interesting as we are looking for a rational function $f\in F(X_1,\dots,X_n,X_{n+1})$ such that 
\[
X_1=f(X_1^m,\dots,X_n^m,T).
\]
(A transposition of $X_1$ and $X_i$ in $f(X_1^m,\dots,X_n^m,T)$ gives $X_i$.) In Section~3, we find such a rational function $f$, and in fact, $f\in F(X_1^m,\dots,X_n^m)[T]$.

\begin{rmk}\label{R1.1}\ \rm
\begin{itemize}
\item[(i)] Each $X_i$, $1\le i\le n$, can be expressed as a rational function in $X_1^m,\dots,X_n^m,T$ over $F$ if and only if 
\begin{equation}\label{1.1}
F(X_1,\dots,X_n)=F(X_1^m,\dots,X_n^m,T).
\end{equation}
By symmetry, \eqref{1.1} holds if and only if $X_1\in F(X_1^m,\dots,X_n^m,T)$.

\item[(ii)] If $n>1$ and $\text{\rm char}\,F=p\mid m$, then \eqref{1.1} does not hold. In fact,
\[
[F(X_1^p,\dots,X_n^p,T):F(X_1^p,\dots,X_n^p)]\le p
\]
since $T^p=X_1^p+\cdots+X_n^p$. On the other hand,
\[
[F(X_1,\dots,X_n):F(X_1^p,\dots,X_n^p)]=p^n>p.
\]
Hence
\[
F(X_1^m,\dots,X_n^m,T)\subset F(X_1^p,\dots,X_n^p,T)\subsetneq F(X_1,\dots,X_n).
\]
\end{itemize}
\end{rmk}

In our notation, $\text{char}\,F$ denotes the characteristic of a field $F$ and $\overline F$ denotes the algebraic closure of $F$. $\f_q$ denotes the finite field with $q$ elements. For $a=(a_1,\dots,a_n)$ and $b=(b_1,\dots,b_n)\in\Bbb Z^n$, $\langle a,b\rangle=\sum_{i=1}^na_ib_i$. The symbol $[a_{ij}]_{i\in I,\,j\in J}$ represents a matrix whose rows are labeled by a set $I$ and columns by a set $J$ and whose $(i,j)$ entry is $a_{ij}$. $\delta_{ij}$ is the Kronecker delta. $s_i(\ )$ is the $i$th elementary symmetric function.

\section{A Nonconstructive Proof}

\begin{thm}\label{T2.1}
Let $F$ be a field and let $T=X_1+\cdots+X_n\in F(X_1,\dots,X_n)$. Let $m$ be a positive integer such that $\text{\rm char}\,F\nmid m$. Then 
\begin{equation}\label{2.1}
F(X_1,\dots,X_n)=F(X_1^m,\dots,X_n^m,T).
\end{equation}
\end{thm}

\begin{proof}
Since $\text{char}\,F\nmid m$, $F(X_1,\dots,X_n)$ is a finite separable extension over $F(X_1^m,\dots,X_n^m)$. Let $K$ be the Galois closure of $F(X_1,\dots,X_n)$ over $F(X_1^m,\dots,X_n^m)$. ($K=F(X_1,\dots,X_n,\epsilon)$, where $\epsilon$ is a primitive $m$th root of unity in $\overline F$.) To prove \eqref{2.1}, it suffices to show that in the Galois correspondence of $K/F(X_1^m,\dots,X_n^m)$, the corresponding groups of $F(X_1,\dots,X_n)$ and $F(X_1^m,\dots,X_n^m,T)$ are the same, i.e., 
\begin{equation}\label{2.2}
\text{Aut}(K/F(X_1,\dots,X_n))=\text{Aut}(K/F(X_1^m,\dots,X_n^m,T)).
\end{equation}
Let $\sigma\in \text{Aut}(K/F(X_1^m,\dots,X_n^m,T))$. For each $1\le i\le n$, we have $\sigma(X_i)=\epsilon^{a_i}X_i$ for some $a_i\in\Bbb Z$. Then 
\[
X_1+\cdots+X_n=T=\sigma(T)=\epsilon^{a_1}X_1+\cdots+\epsilon^{a_n}X_n.
\]
It follows that for all $1\le i\le n$, $\epsilon^{a_i}=1$, that is, $\sigma(X_i)=X_i$. Therefore $\sigma\in \text{Aut}(K/F(X_1,\dots,X_n))$ and \eqref{2.2} is proved.
\end{proof}

Assume that $\text{char}\,F\nmid m$ and let $\epsilon\in\overline F$ be a primitive $m$th root of unity. By Theorem~\ref{T2.1}, 
\[
[F(X_1^m,\dots,X_n^m,T):F(X_1^m,\dots,X_n^m)]=[F(X_1,\dots,X_n):F(X_1^m,\dots,X_n^m)]=m^n,
\]
that is, $T$ is of degree $m^n$ over $F(X_1^m,\dots,X_n^m)$. On the other hand, every conjugate of $T$ over $F(X_1^m,\dots,X_n^m)$ is of the form $\epsilon^{a_1}X_1+\cdots+\epsilon^{a_n}X_n$, where $(a_1,\dots,a_n)\in\{0,1,\dots,m-1\}^n$. Hence the minimal polynomial of $T$ over $F(X_1^m,\dots,X_n^m)$ is
\[
\prod_{(a_1,\dots,a_n)\in\{0,1,\dots,m-1\}^n}\Bigl(X-\sum_{i=1}^n\epsilon^{a_i}X_i\Bigr).
\]

\begin{rmk}\label{R2.2}\rm
In Theorem~\ref{T2.1}, $X_1,\dots,X_n$ can be replaced by elements $x_1,\dots,x_n$ in an extension of $F$. As long as $x_1+\cdots+x_n=\epsilon^{a_1}x_1+\cdots+\epsilon^{a_n}x_n$ implies $\epsilon^{a_i}=1$ for all $i$, Theorem~\ref{2.1} holds. The result in this form frequently appears as exercises in graduate algebra textbooks; see, for example, \cite[Exercise~18.14]{Isaacs-ams-gsm-2009}.
\end{rmk}

\section{A Constructive Proof}

Assume that $\text{char}\,F\nmid m$. Our objective is to express $X_1$ as a polynomial in $T=X_1+\cdots+X_n$ with coefficients in $F(X_1^m,\dots,X_n^m)$.

\subsection{A naive approach}\

Let's begin with an example to show what a naive approach might look like.

\begin{exmp}\label{E3.1}
\rm
Let $m=2$. We try to express $X_1$ as a rational function in $X_1^2,\dots,X_n^2,T$.

When $n=1$, $X_1=T$.

When $n=2$,
\[
T^2=X_1^2+X_2^2+2X_1X_2=X_1^2+X_2^2+2X_1(T-X_1)=-X_1^2+X_2^2+2TX_1,
\]
whence
\[
X_1=\frac{T^2+X_1^2-X_2^2}{2T}.
\]

When $n=3$, from $(T-X_1)^2=(X_2+X_3)^2$, we have
\begin{align*}
&T^2+X_1^2-2TX_1=X_2^2+X_3^2+2X_2X_3,\cr
&T^2+X_1^2-X_2^2-X_3^2-2TX_1=2X_2X_3,\cr
&(T^2+X_1^2-X_2^2-X_3^2-2TX_1)^2=4X_2^2X_3^2,\cr
&(T^2+X_1^2-X_2^2-X_3^2)^2+4T^2X_1^2-4T(T^2+X_1^2-X_2^2-X_3^2)X_1=4X_2^2X_3^2,\cr
&X_1=\frac{(T^2+X_1^2-X_2^2-X_3^2)^2+4T^2X_1^2-4X_2^2X_3^2}{4T(T^2+X_1^2-X_2^2-X_3^2)}.
\end{align*}

When $n=4$, from $(T-X_1-X_2)^2=(X_3+X_4)^2$, we have
\[
(T-X_1)^2+X_2^2-2(T-X_1)X_2=X_3^2+X_4^2+2X_3X_4.
\]
In what follows, $f_1,\dots,f_5$ are polynomials in $X_1^2,\dots,X_n^2, T$ over $F$. We have
\begin{align*}
&-2TX_1+f_1=2(T-X_1)X_2+2X_3X_4,\cr
&(-TX_1+f_1/2)^2=((T-X_1)X_2+X_3X_4)^2,\cr
&-Tf_1X_1+f_2=(T-X_1)^2X_2^2+2(T-X_1)X_2X_3X_4,\cr
&(2TX_2^2-Tf_1)X_1+f_3=2(T-X_1)X_2X_3X_4,\cr
&(T(2X_2^2-f_1)X_1+f_3)^2=(2(T-X_1)X_2X_3X_4)^2,\cr
&2T(2X_2^2-f_1)f_3X_1+f_4=4(T-X_1)^2X_2^2X_3^2X_4^2,\cr
&(2T(2X_2^2-f_1)f_3+8TX_2^2X_3^2X_4^2)X_1=f_5,\cr
&X_1=\frac{f_5}{2T((2X_2^2-f_1)f_3+4X_2^2X_3^2X_4^2)}.
\end{align*}

Unfortunately, when $n\ge 5$, it appears that this method no longer works.
\end{exmp}

\subsection{The case $\text{char}\,\boldsymbol{F=p>0}$}\

When $\text{char}\,F=p>0$, we can take advantage of the fact that $T^p=X_1^p+\cdots+X_n^p$. Since $p\nmid m$, there exists a positive integer  $e$ such that $p^e\equiv 1\pmod m$. Let $q=p^e$. The matrix 
\[
M=\left[
\begin{matrix}
X_1^{q^0}&\cdots&X_n^{q^0}\cr
\vdots&&\vdots\cr
X_1^{q^{n-1}}&\cdots&X_n^{q^{n-1}}
\end{matrix}
\right]
\]
is called the {\em Moore Matrix}. Its determinant, the {\em Moore determinant}, denoted by $\Delta(X_1,\dots,X_n)$, is given in \cite{Moore-BAMS-1896} by
\[
\Delta(X_1,\dots,X_n)=\det M=\prod_{i=1}^n\,\prod_{a_1,\dots,a_{i-1}\in\f_q}\Bigl(X_i+\sum_{j=1}^{i-1}a_jX_j\Bigr);
\]
also see \cite[Exercise~2.15]{Hou-ams-gsm-2018} or \cite[Lemma~3.51]{Lidl-Niederreiter-FF-1997}. (When $i=1$, the inner product in the above is understood to be $X_1$.) We have
\[
M\left[
\begin{matrix}
X_1^{-1}\cr 
&\ddots\cr
&&X_n^{-1}
\end{matrix}
\right]
\left[
\begin{matrix}
X_1\cr
\vdots\cr
X_n
\end{matrix}
\right]=
\left[
\begin{matrix}
T^{q^0}\cr
\vdots\cr
T^{q^{n-1}}
\end{matrix}
\right],
\]
where the entries of 
\[
M\left[
\begin{matrix}
X_1^{-1}\cr 
&\ddots\cr
&&X_n^{-1}
\end{matrix}
\right]
\]
belong to $F[X_1^{q-1},\dots,X_n^{q-1}]\subset F[X_1^m,\dots,X_n^m]$. Thus
\[
\left[
\begin{matrix}
X_1\cr
\vdots\cr
X_n
\end{matrix}
\right]=
\left[
\begin{matrix}
X_1\cr 
&\ddots\cr
&&X_n
\end{matrix}
\right]M^{-1}
\left[
\begin{matrix}
T^{q^0}\cr
\vdots\cr
T^{q^{n-1}}
\end{matrix}
\right],
\]
where the entries of 
\[
\left[
\begin{matrix}
X_1\cr 
&\ddots\cr
&&X_n
\end{matrix}
\right]M^{-1}
\]
belong to $F(X_1^{q-1},\dots,X_n^{q-1})$. We have
\[
M^{-1}=\frac 1{\Delta(X_1,\dots,X_n)}\left[\begin{matrix}\Delta_0&-\Delta_1&\cdots&(-1)^{n-1}\Delta_{n-1}\cr
\cr
&&*\cr
\cr
\end{matrix}\right],
\]
where
\begin{equation}\label{delta_i}
\Delta_i=\left|\begin{matrix}X_2^{q^0}&\cdots&X_n^{q^0}\cr
\vdots&&\vdots\cr
X_2^{q^{i-1}}&\cdots&X_n^{q^{i-1}}\cr
X_2^{q^{i+1}}&\cdots&X_n^{q^{i+1}}\cr
\vdots&&\vdots\cr
X_2^{q^{n-1}}&\cdots&X_n^{q^{n-1}}
\end{matrix}\right|,\qquad 0\le i\le n-1.
\end{equation}
Hence
\begin{equation}\label{X1}
X_1=\sum_{i=0}^{n-1}\frac{(-1)^iX_1\Delta_i}{\Delta(X_1,\dots,X_n)}T^{q^i},
\end{equation}
where 
\[
\frac{X_1\Delta_i}{\Delta(X_1,\dots,X_n)}\in F(X_1^{q-1},\dots,X_n^{q-1}).
\]
The determinant $\Delta_i$ can be computed by using a normal basis of $\f_{q^n}$ over $\f_q$ or by using the roots of any irreducible polynomial of degree $n$ over $\f_q$; see the appendix for the details.

\subsection{The general case}\

Let $\Lambda=\{0,1,\dots,m-1\}^n$. For $t\ge 0$,
\[
T^t=\sum_{i_1,\dots,i_n}\binom t{i_1,\dots,i_n}X_1^{i_1}\cdots X_n^{i_n}=\sum_{\lambda=(\lambda_1,\dots,\lambda_n)\in\Lambda}a(t,\lambda)X_1^{\lambda_1}\cdots X_n^{\lambda_n},
\]
where
\begin{align*}
a(t,\lambda)\,&=X_1^{-\lambda_1}\cdots X_n^{-\lambda_n}\sum_{(i_1,\dots,i_n)\equiv\lambda\,({\rm mod}\, m)}\binom t{i_1,\dots,i_n}X_1^{i_1}\cdots X_n^{i_n}\cr
&\in F(X_1^m,\dots,X_n^m).
\end{align*}
So, we have a linear system in $X_1^{\lambda_1}\cdots X_n^{\lambda_n}$, $\lambda\in \Lambda$:
\[
\sum_{\lambda\in\Lambda}a(t,\lambda)X_1^{\lambda_1}\cdots X_n^{\lambda_n}=T^t,\quad 0\le t\le m^n-1,
\]
whose coefficient matrix is 
\[
A=[a(t,\lambda)]_{0\le t\le m^n-1,\,\lambda\in\Lambda}.
\]
Let $\epsilon\in\overline F$ be a primitive $m$th root of unity. Since
\[
\sum_{i=0}^{m-1}\epsilon^{ij}=\begin{cases}
0&\text{if}\ j\not\equiv 0\pmod m,\cr
m&\text{if}\ j\equiv 0\pmod m,
\end{cases}
\]
we have
\begin{align*}
&a(t,\lambda)\cr
&=X_1^{-\lambda_1}\cdots X_n^{-\lambda_n}\sum_{i_1,\dots,i_n}\binom t{i_1,\dots,i_n}X_1^{i_1}\cdots X_n^{i_n}\prod_{k=1}^n\Bigl(\frac 1m\sum_{j_k=0}^{m-1}\epsilon^{(i_k-\lambda_k)j_k}\Bigr)\cr
&=m^{-n}X_1^{-\lambda_1}\cdots X_n^{-\lambda_n}\sum_{i_1,\dots,i_n}\binom t{i_1,\dots,i_n}X_1^{i_1}\cdots X_n^{i_n}\prod_{k=1}^n\Bigl(\sum_{j_k=0}^{m-1}\epsilon^{(i_k-\lambda_k)j_k}\Bigr)\cr
&=m^{-n}X_1^{-\lambda_1}\cdots X_n^{-\lambda_n}\sum_{i_1,\dots,i_n}\binom t{i_1,\dots,i_n}X_1^{i_1}\cdots X_n^{i_n}\sum_{(j_1,\dots,j_n)\in\Lambda}\epsilon^{\langle i-\lambda,j\rangle}\cr
&\kern17.5em (i=(i_1,\dots,i_n),\ j=(j_1,\dots,j_n))\cr
&=m^{-n}X_1^{-\lambda_1}\cdots X_n^{-\lambda_n}\sum_{j\in\Lambda}\epsilon^{-\langle\lambda,j\rangle}\sum_{i_1,\dots,i_n}\binom t{i_1,\dots,i_n}X_1^{i_1}\cdots X_n^{i_n}\epsilon^{\langle i,j\rangle}\cr
&=m^{-n}X_1^{-\lambda_1}\cdots X_n^{-\lambda_n}\sum_{j\in\Lambda}\epsilon^{-\langle\lambda,j\rangle}\sum_{i_1,\dots,i_n}\binom t{i_1,\dots,i_n}(\epsilon^{j_1}X_1)^{i_1}\cdots(\epsilon^{j_n}X_n)^{i_n}\cr
&=m^{-n}X_1^{-\lambda_1}\cdots X_n^{-\lambda_n}\sum_{j\in\Lambda}\epsilon^{-\langle\lambda,j\rangle}\Bigl(\sum_{k=1}^n\epsilon^{j_k}X_k\Bigr)^t.
\end{align*}
Therefore
\[
A=m^{-n}BCD,
\]
where
\[
B=\Bigl[\Bigl(\sum_{k=1}^n\epsilon^{j_k}X_k\Bigr)^t\Bigr]_{0\le t\le m^n-1,\, j\in\Lambda}
\]
is a nonsingular Vandermonde matrix,
\[
C=\bigl[\epsilon^{-\langle j,\lambda\rangle}\bigr]_{j\in\Lambda,\,\lambda\in\Lambda},
\]
and $D$ is the $m^n\times m^n$ diagonal matrix labeled by $\Lambda\times\Lambda$ whose $(\lambda,\lambda)$ entry is $X_1^{-\lambda_1}\cdots X_n^{-\lambda_n}$. From the orthogonality of the characters of $(\Bbb Z/m\Bbb Z)^n$, we have
\[
C^{-1}=m^{-n}\bigl[\epsilon^{\langle\lambda,j\rangle}\bigr]_{\lambda\in\Lambda\,j\in\lambda}.
\]
Hence
\begin{equation}\label{3.1}
A^{-1}=m^nD^{-1}C^{-1}B^{-1}=D^{-1}\bigl[\epsilon^{\langle\lambda,j\rangle}\bigr]_{\lambda\in\Lambda\,j\in\lambda} B^{-1}.
\end{equation}
Let
\[
b_j=\sum_{k=1}^n\epsilon^{j_k}X_k,\qquad j=(j_1,\dots,j_n)\in\Lambda.
\]
Then
\[
B=\bigl[b_j^t\bigr]_{0\le t\le m^n-1,\, j\in\Lambda}.
\]
Although the inverse of a nonsingular Vandermonde matrix is well known \cite{Macon-Spitzbart-AMM-1958}, we include a brief derivation of $B^{-1}$ for the reader's conevenience. Let 
\begin{equation}\label{3.2}
B^{-1}=[a_{jt}]_{j\in\Lambda,\,0\le t\le m^n-1}.
\end{equation}
Then 
\[
\sum_{t=0}^{m^n-1}a_{jt}b^t_i=\delta_{ji},\qquad i\in\Lambda.
\]
Hence 
\[
\sum_{t=0}^{m^n-1}a_{jt}X^t=\frac{\prod_{i\ne j}(X-b_i)}{\prod_{i\ne j}(b_j-b_i)}.
\]
So
\begin{equation}\label{3.3}
a_{jt}=(-1)^{m-1-t}\,\frac{s_{m^n-1-t}((b_i)_{i\ne j})}{\prod_{i\ne j}(b_j-b_i)},
\end{equation}
where $s_{m^n-1-t}(\ )$ is the $(m^n-1-t)$-th elementary symmetric function. Now by \eqref{3.1} and \eqref{3.2}, we have
\[
X_1=X_1\sum_{j\in\Lambda}\sum_{t=0}^{m^n-1}\epsilon^{j_1}a_{jt}T^t=\sum_{t=0}^{m^n-1}\Bigl(X_1\sum_{j\in\Lambda}\epsilon^{j_1}a_{jt}\Bigr)T^t,
\]
where $X_1\sum_{j\in\Lambda}\epsilon^{j_1}a_{jt}\in F(X_1^m,\dots,X_n^m)$.

\section*{Appendix. Computation of $\Delta_i$}

Let $\Delta_i$ be the determinant in \eqref{delta_i}. We show that it can be computed in two ways: by using a normal basis of $\f_{q^n}$ over $\f_q$ or by using the roots of an irreducible polynomial of degree $n$ over $\f_q$. 

\subsection*{A1. Computation of $\boldsymbol{\Delta_i}$ using a normal basis of $\boldsymbol{\f_{q^n}}$ over $\boldsymbol{\f_q}$}\

For $z\in\f_{q^n}$, define
\[
M(z)=\left[\begin{matrix}
z&z^q&\cdots&z^{q^{n-1}}\cr
z^q&z^{q^2}&\cdots&z\cr
\vdots&&\vdots\cr
z^{q^{n-1}}&z&\cdots&z^{q^{n-2}}
\end{matrix}
\right].
\]
Then $z,z^q,\dots,z^{q^{n-1}}$ form a normal basis of $\f_{q^n}$ over $\f_q$ if and only if $\det M(z)=\Delta(z,z^q,\dots,z^{q^{n-1}})\ne 0$. Choose $z\in\f_{q^n}$ such that $\det M(z)\ne 0$. Then $M(z)^{-1}=M(w)$ for some $\in\f_{q^n}$; see \cite[\S2]{Bartoli-Hou-arXiv2008.03432}. Since
\[
M(z)^{-1}=\frac 1{\det M(z)}\left[\begin{matrix}a_{11}&\cdots\cr \vdots\end{matrix}\right],
\]
where 
\begin{align*}
a_{11}\,&=\left|\begin{matrix}
z^{q^2}&z^{q^3}&\cdots&z\cr
z^{q^3}&z^{q^4}&\cdots&z^q\cr
\vdots&\vdots&&\vdots\cr 
z&z^q&\cdots&z^{q^{n-2}}\end{matrix}\right|=
\left|\begin{matrix}
z&z^{q}&\cdots&z^{q^{n-2}}\cr
z^{q}&z^{q^2}&\cdots&z^{q^{n-1}}\cr
\vdots&\vdots&&\vdots\cr 
z^{q^{n-2}}&z^{q^{n-1}}&\cdots&z^{q^{n-4}}\end{matrix}\right|^{q^2}\cr
&=\Delta(z,z^q,\dots,z^{q^{n-2}})^{q^2},
\end{align*}
We have 
\[
w=\frac{a_{11}}{\det M(z)}=\frac{\Delta(z,z^q,\dots,z^{q^{n-2}})^{q^2}}{\Delta(z,z^q,\dots,z^{q^{n-1}})}.
\]
From $M(w)M(z)=I$, we have
\[
\sum_{k=0}^{n-1}(w^{q^i})^{q^k}(z^{q^k})^{q^j}=\delta_{ij}.
\]
Hence
\begin{align*}
\Delta_i\,&=(-1)^i\left|\begin{matrix}\sum_{k=0}^{n-1}(w^{q^i})^{q^k}(z^{q^k})^{q^0}&X_2^{q^0}&\cdots&X_n^{q^0}\cr
\vdots&\vdots&&\vdots\cr
\sum_{k=0}^{n-1}(w^{q^i})^{q^k}(z^{q^{k}})^{q^{n-1}}&X_2^{q^{n-1}}&\cdots& X_n^{q^{n-1}}\end{matrix}\right|\cr
&=(-1)^i\sum_{k=0}^{n-1}(w^{q^i})^{q^k}\left|\begin{matrix}(z^{q^k})^{q^0}&X_2^{q^0}&\cdots&X_n^{q^0}\cr
\vdots&\vdots&&\vdots\cr
(z^{q^k})^{q^{n-1}}&X_2^{q^{n-1}}&\cdots& X_n^{q^{n-1}}\end{matrix}\right|\cr
&=(-1)^i\sum_{k=0}^{n-1}w^{q^{i+k}}\Delta(z^{q^k},X_2,\dots,X_n)\cr
&=(-1)^i\sum_{k=0}^{n-1}\frac{\Delta(z,z^q,\dots,z^{q^{n-2}})^{q^{i+k+2}}}{\Delta(z,z^q,\dots,z^{q^{n-1}})^{q^{i+k}}}\Delta(z^{q^k},X_2,\dots,X_n).
\end{align*}

\subsection*{A2. Computation of $\boldsymbol{\Delta_i}$ using the roots of an irreducible polynomial of degree $\boldsymbol n$ over $\boldsymbol{\f_q}$}\

Let $\alpha\in\f_{q^n}$ be of degree $n$ over $\f_q$ and let $\alpha_i=\alpha^{q^i}$, $0\le i\le n-1$. Let
\begin{equation}\tag{A1}\label{g}
g(X)=\frac{\prod_{i=1}^{n-1}(X-\alpha_i)}{\prod_{i=1}^{n-1}(\alpha_0-\alpha_i)}=a_0+\cdots+a_{n-1}X^{n-1},
\end{equation}
which satisfies $g(\alpha_i)=\delta_{i0}$, $0\le i\le n-1$. Then
\[
\left[\begin{matrix}\alpha^0&\alpha^1&\cdots&\alpha^{n-1}\cr
\alpha^{0q}&\alpha^{1q}&\cdots&\alpha^{(n-1)q}\cr
\vdots&\vdots&&\vdots\cr
\alpha^{0q^{n-1}}&\alpha^{1q^{n-1}}&\cdots&\alpha^{(n-1)q^{n-1}}\end{matrix}\right]
\left[\begin{matrix}a_0\cr\vdots\cr a_{n-1}\end{matrix}\right]=\left[\begin{matrix}1\cr 0\cr\vdots\cr 0\end{matrix}\right].
\]
It follows that
\[
\left[\begin{matrix}\alpha^0&\alpha^1&\cdots&\alpha^{n-1}\cr
\alpha^{0q}&\alpha^{1q}&\cdots&\alpha^{(n-1)q}\cr
\vdots&\vdots&&\vdots\cr
\alpha^{0q^{n-1}}&\alpha^{1q^{n-1}}&\cdots&\alpha^{(n-1)q^{n-1}}\end{matrix}\right]^{-1}=\left[\begin{matrix} a_0&a_0^q&\cdots&a_0^{q^{n-1}}\cr
\vdots&\vdots&&\vdots\cr
a_{n-1}&a_{n-1}^q&\cdots&a_{n-1}^{q^{n-1}}\end{matrix}\right],
\]
and hence
\[
\sum_{k=0}^{n-1}\alpha^{kq^i}a_k^{q^j}=\delta_{ij}.
\]
Therefore
\begin{align}\tag{A2}\label{A2}
\Delta_i\,&=(-1)^i\left|\begin{matrix}\sum_{k=0}^{n-1}\alpha^{kq^i}a_k^{q^0}&X_2^{q^0}&\cdots&X_n^{q^0}\cr
\vdots&\vdots&&\vdots\cr
\sum_{k=0}^{n-1}\alpha^{kq^i}a_k^{q^{n-1}}&X_2^{q^{n-1}}&\cdots&X_n^{q^{n-1}}\end{matrix}\right|\\
&=(-1)^i\sum_{k=0}^{n-1}\alpha^{kq^i}\left|\begin{matrix}a_k^{q^0}&X_2^{q^0}&\cdots&X_n^{q^0}\cr
\vdots&\vdots&&\vdots\cr
a_k^{q^{n-1}}&X_2^{q^{n-1}}&\cdots&X_n^{q^{n-1}}\end{matrix}\right|\cr
&=(-1)^i\sum_{k=0}^{n-1}\alpha^{kq^i}\Delta(a_k,X_2,\dots,X_n).\nonumber
\end{align}

In \eqref{g} and \eqref{A2},
\[
a_k=\frac{(-1)^{n-1-k}s_{n-i-k}(\alpha_1,\dots,\alpha_{n-1})}{\prod_{i=1}^{n-1}(\alpha_0-\alpha_i)}.
\]
However, it is more effective to compute $a_k$ in terms of the minimal polynomial $f(X)$ of $\alpha$ over $\f_q$. Let
\[
f(X)=\prod_{i=0}^{n-1}(X-\alpha_i)=X^n+b_{n-1}X^{n-1}+\cdots+b_0.
\]
Then
\[
f'(\alpha)=\prod_{i=1}^{n-1}(\alpha-\alpha_i)
\]
and 
\[
f'(\alpha)(X-\alpha)(a_0+\cdots+a_{n-1}X^{n-1})=b_0+\cdots+b_{n-1}X^{n-1}+X^n,
\]
that is,
\[
f'(\alpha)\left[\begin{matrix}-\alpha\cr
1&-\alpha\cr
&\cdot&\cdot\cr
&&\cdot&\;\cdot\cr
&&&\;\cdot&\;\cdot\cr
&&&&1&-\alpha\end{matrix}\right]
\left[\begin{matrix} a_0\cr\vdots\cr a_{n-1}\end{matrix}\right]=\left[\begin{matrix} b_0\cr\vdots\cr b_{n-1}\end{matrix}\right].
\]
In the above, the inverse of the $n\times n$ matrix is
\[
-\left[\begin{matrix}
\alpha^{-1}\cr
\alpha^{-2}&\alpha^{-1}\cr
\cdot&\cdot&\cdot\cr
\cdot&&\cdot&\ \cdot\cr
\cdot&&&\ \cdot&\cdot\cr
\alpha^{-n}&\cdot&\cdot&\ \cdot&\alpha^{-2}&\alpha^{-1}\end{matrix}\right].
\]
Hence
\begin{equation}\tag{A3}\label{A3}
a_k=\frac{-1}{f'(\alpha)}\sum_{j=0}^k\alpha^{-k-1+j}b_j.
\end{equation}
Making substitution \eqref{A3} in \eqref{A2} gives
\begin{align*}
\Delta_i\,&=(-1)^i\sum_{k=0}^{n-1}\alpha^{kq^i}\Delta\Bigl(\frac{-1}{f'(\alpha)}\sum_{j=0}^k\alpha^{-k-1+j}b_j,X_2,\dots,X_n\Bigr)\cr
&=(-1)^{n+i}\sum_{k=0}^{n-1}\alpha^{kq^i}\sum_{j=0}^k b_j\Delta\Bigl(\frac{\alpha^{-k-1+j}}{f'(\alpha)},X_2,\dots,X_n\Bigr).
\end{align*}



\end{document}